\documentclass{article}
\usepackage[left=2.50cm, right=2.50cm, top=2.50cm, bottom=2.50cm]{geometry} 
\usepackage{helvet}
\usepackage{amsmath, amsfonts, amssymb} 
\usepackage[english]{babel}
\usepackage{graphicx}   
\usepackage{url}        
\usepackage{bm}      
\usepackage{multirow}
\usepackage{booktabs}
\usepackage{titlesec}
\usepackage{amsmath}
\usepackage{mathrsfs}
\usepackage{amsmath, amsthm}
\usepackage{cite}
\usepackage{verbatim}
\usepackage{todonotes}
\usepackage[table]{xcolor} 

\newcommand{\adef}{\leftarrow}

\usepackage{algorithm}
\usepackage{algpseudocode}
\usepackage{float}
\usepackage{mathtools}

\algrenewcommand\algorithmicwhile{\textbf{While}}
\algrenewcommand\algorithmicfor{\textbf{For}}
\algrenewcommand\algorithmicdo{\textbf{Do}}
\algrenewcommand\algorithmicif{\textbf{If}}
\algrenewcommand\algorithmicthen{\textbf{Then}}
\algrenewcommand\algorithmicelse{\textbf{Else}}
\algrenewcommand\algorithmicend{\textbf{End}}
\algrenewcommand\algorithmicreturn{\textbf{Return}}

\titleformat{\section}
  {\normalfont\Large\bfseries}{\llap{\thesection\quad}}{0pt}{}

\newtheorem{definition}{Definition}[section]
\newtheorem{theorem}{Theorem}[section]
\newtheorem{corollary}{Corollary}[section]
\newtheorem{remark}{Remark}[section]
\newtheorem{lemma}{Lemma}[section]
\newtheorem{example}{Example}[section]
\newtheorem{proposition}{Proposition}[section]
\newtheorem*{theorem*}{Theorem}
\newtheorem*{maintheorema}{Main Theorem A}
\newtheorem*{maintheoremb}{Main Theorem B}

\newcommand{\Hq}{\mathbb{H}}
\newcommand{\R}{\mathbb{R}}

\newcommand{\qi}{\mathbf{i}}
\newcommand{\qj}{\mathbf{j}}
\newcommand{\qk}{\mathbf{k}}

\newcommand{\Cj}[1]{{#1}^\ast}

\newcommand{\QNorm}[1]{\nu(#1)}

\usepackage{fancyhdr} 
\usepackage{hyperref} 
\hypersetup{colorlinks, unicode} 
\usepackage{multicol}
\title{\textbf{Kempe factorizations for rational curves on $\operatorname{SO}_4(\mathbb{R})$}}
\author{\sffamily Yifan Li, \quad \sffamily Zijia Li, \quad \sffamily Ke Ye}
\date{\today}
\begin{document}
	\maketitle
\begin{abstract}
We study constructive Kempe factorizations for rational curves on $\operatorname{SO}_4(\mathbb{R})$. Motivated by motion-polynomial factorization and rational matrix curves on real classical groups, we prove that every rational curve of degree $2d$ with $d\ge1$ first factors into $d$ quadratic rational curves and then into a product of at most $2d$ planar rotation curves, where each planar rotation curve fixes a two-dimensional plane pointwise. The construction proceeds by extracting left and right isoclinic polynomial parts via Cayley's factorization, factoring the corresponding quaternion polynomials, and pairing linear factors with equal norm polynomials. The resulting algorithms are explicit and are illustrated by examples.
\end{abstract}
	
\section{Introduction}

\subsection{Kempe factorization and rational motion synthesis}

Kempe's Universality Theorem~\cite{Kempe75} shows that bounded segments of planar algebraic curves can be traced by linkages with revolute joints. Its constructive nature makes it a foundational result in mechanism theory, but the original construction is far from economical: even low-degree curves may require a large number of links and joints. This gap between existence and efficient construction motivates a more algebraic question: when an algebraic motion is given, can it be decomposed into simple motion primitives of controlled degree?

In this paper we restrict the problem to rational curves. This restriction is motivated by two standard features of rational parametrizations. On the algebraic side, rational curves are precisely the algebraic curves admitting explicit parametrizations by rational functions, and they form the natural class for symbolic parametrization and factorization algorithms \cite{SW99,SWP08,Gordon65zero}. On the kinematic side, rational motions are compatible with the rational B\'{e}zier/NURBS representations used in computer-aided geometric design and mechanism synthesis \cite{PW10}. We therefore do not attempt to decompose arbitrary algebraic curves in matrix groups, but focus on rational curves and require all factors in the factorization to remain rational. It is worth emphasizing that the rational curves considered here are rational matrix curves in a real algebraic group arising from motion factorization, rather than rational curves on projective varieties in the usual algebraic-geometric sense.

For rational motions in three-dimensional Euclidean space, motion-polynomial factorization provides a powerful method. A motion polynomial can be factored into linear factors, each corresponding to a rotation about a fixed axis in $\mathbb{R}^3$; geometrically, these factors correspond to lines on the Study quadric~\cite{Hegedus13,Scharler17}. This point of view underlies Kempe-type factorizations for rational curves~\cite{LSS19,GKLRSV17} and has been developed further in geometric algebra and linkage synthesis~\cite{Hegedus13,LJS18,Thimm2026}. A matrix-based formulation also extends the problem to rational curves on homogeneous spaces~\cite{Liye2024rational}.

Compared with these motion-polynomial and homogeneous-space approaches, this paper focuses on the factorization inside $\operatorname{SO}_4(\mathbb{R})$. This is not a direct extension of motion-polynomial factorization in $\operatorname{SE}_3$. In the matrix model for $\operatorname{SO}_4(\mathbb R)$, a linear quaternion factor naturally gives rise to a quadratic matrix curve. Thus the basic building blocks in our factorization are quadratic matrix curves rather than linear ones. The gain is that the left--right isoclinic splitting reduces the construction to ordinary quaternion-polynomial factorization and makes the planar-rotation refinement explicit. For convenience, we compare the existing methods with ours in Table~\ref{tab:method-comparison}.

\begin{table}[htbp]
\centering
\caption{Comparison with related factorization methods}
\label{tab:method-comparison}
\begin{tabular}{p{0.27\textwidth}p{0.4\textwidth}p{0.15\textwidth}}
\toprule
Method & Ambient space & Primitive factor \\
\midrule
Motion-polynomial \cite{GKLRSV17,LJS18,Li25geometric}
& $\operatorname{SE}_2$, $\operatorname{SE}_3$, $\operatorname{SO}_{3,1}$, etc.
& linear \\
\midrule
Rational curves \cite{Liye2024rational}
& general homogeneous spaces
& quadratic \\
\midrule
Cayley's factorization
& $\operatorname{SO}_4(\mathbb{R})$
& planar motion \\
\bottomrule
\end{tabular}
\end{table}

\subsection[Why SO(4)]{Why $\operatorname{SO}(4)$}

The four-dimensional rotation group is a natural next testing ground for this constructive theory. On the one hand, $\operatorname{SO}_4(\mathbb{R})$ is already rich enough to contain two independent rotation planes. Thus, its low-degree factors are not simply the same as the one-axis rotations in $\operatorname{SO}_3(\mathbb{R})$. On the other hand, $\operatorname{SO}_4(\mathbb{R})$ has a special algebraic structure: every rotation admits a Cayley factorization into left and right isoclinic components, reflecting the quaternionic description of $\operatorname{Spin}(4)$. This makes $\operatorname{SO}_4(\mathbb{R})$ a particularly suitable setting for translating rational matrix curves into quaternion polynomial factorization.

This choice is also supported by the existing four-dimensional kinematics literature. Quaternionic models, kinematic mappings, and hypersphere conditions for motions in four dimensional Euclidean space have been studied in several closely related forms~\cite{Nawratil2014kinematic,Nawratil2016quaternionic,Nawratil2016fundamentals}. In the rotation-only case, the factorization of a four-dimensional rotation into isoclinic components is precisely Cayley's factorization of $4$D rotations~\cite{Perez17cayley}; see also the quaternionic account of rotations in $\mathbb{R}^4$ in~\cite{Weiner05quaternion}. From the Lie-theoretic side, this is the familiar special feature that 
\[\operatorname{Spin}(4) \cong \operatorname{Spin}(3) \times \operatorname{Spin}(3) \cong \mathbb{H}_1 \times \mathbb{H}_1,
\]
a perspective compatible with the Clifford and Lie group background in~\cite{lounesto2001clifford,hall2013lie}. Here $\mathbb H_1$ denotes the group of unit quaternions, and $\operatorname{Spin}(n)$ is the spin group, which is the double cover of $\operatorname{SO}_n(\mathbb{R})$.

The main difference from classical motion-polynomial factorization \cite{GKLRSV17,LJS18,LSS19,Thimm2026} is that we work directly with rational curves in the matrix representation of $\operatorname{SO}_4(\mathbb{R})$. This matrix viewpoint matches the general framework of rational curves on real classical groups~\cite{Liye2024rational}. In this representation the natural primitive degree is two rather than one. Indeed, a linear quaternion factor gives rise, after passing to the matrix representation, to a quadratic rational curve. Thus the first meaningful factorization problem is not into linear matrix curves, but into quadratic rational curves and then into planar rotation curves. In this sense, $\operatorname{SO}_4(\mathbb{R})$ is the first setting in which the quaternionic factorization, the matrix model, and a genuinely planar rotation primitive meet in a nontrivial but still explicitly controllable way.

\subsection{Main results}
The main result of the paper is a factorization into planar rotations. A planar rotation curve is a quadratic rational curve whose values rotate one fixed two-dimensional plane and fix its orthogonal complement pointwise. Informally, the main theorem is the following.

\begin{maintheorema}[planar-rotation factorization; see Theorem~\ref{thm:futherdecomp}]\label{thm:mainA}
Every rational curve on $\operatorname{SO}_4(\mathbb{R})$ of degree $2d$ admits a factorization into a product of at most $2d$ planar rotation curves.
\end{maintheorema}

The proof proceeds through a key intermediate factorization theorem, which is also of independent algebraic interest. In the matrix model, the primitive nontrivial rational factors have degree two; we call such factors quadratic rational curves.

\begin{maintheoremb}[quadratic factorization; see Theorem~\ref{thm:decomp}] 
On $\operatorname{SO}_4(\mathbb{R})$, every rational curve  of degree $2d$ is a product of $d$ quadratic rational curves.
\end{maintheoremb}

This quadratic factorization theorem is the $\operatorname{SO}_4(\mathbb{R})$ analogue of the first stage of motion-polynomial factorization: the curve is decomposed into factors of the smallest degree in the matrix model. Each quadratic factor is then shown to be a product of at most two planar rotation curves.

\subsection{Proof strategy and algorithmic output}

The proof of Theorem~\ref{thm:futherdecomp} has three steps. First, we apply Cayley's factorization to extract right and left isoclinic parts from an arbitrary rational curve on $\operatorname{SO}_4(\mathbb{R})$. Second, we identify these isoclinic polynomial matrices with quaternion polynomials and factor them into linear quaternion factors. Third, we pair linear factors with equal norm polynomials; each pair gives a quadratic rational curve on $\operatorname{SO}_4(\mathbb{R})$, and the three possible types of quadratic factors are then decomposed into one or two planar rotation curves.

This proof is constructive throughout. Algorithm~\ref{alg:decomp} produces the factorization into quadratic factors, and Algorithm~\ref{alg:plane_rotation_decomp} refines it into planar rotations. The examples in Section~\ref{sec:example} are included to make the full pipeline visible: Cayley's factorization, quaternion factorization, pairing of equal-norm factors, and planar-rotation refinement.

\section{Preliminaries}

This section fixes notation and recalls the two algebraic tools used throughout the paper. The first tool is the Cayley factorization of four-dimensional rotations into left and right isoclinic factors. The second is the factorization of quaternion polynomials into linear factors. Together they allow us to turn a factorization problem for rational matrix curves on $\operatorname{SO}_4(\mathbb{R})$ into a factorization problem in $\mathbb{H}[t]$.

\subsection[Rational curves on SO4]{Rational curves on $\operatorname{SO}_4(\mathbb{R})$}\label{subsec:curveSO4}

We view the special orthogonal group $\operatorname{SO}_4(\mathbb{R})$ as the group consisting of $X \in \mathbb{R}^{4 \times 4}$ such that $XX^\top = I_4$ and $\det (X) = 1$.

\begin{definition}\label{def:ratcurve}
A rational curve on $\operatorname{SO}_4(\mathbb{R})$ is a map $\gamma: \mathbb{R}^1 \to \operatorname{SO}_4(\mathbb{R})$ of the form $\gamma(t) = \big( p_{ij}(t)/q(t) \big)_{1 \le i,j \le 4}$ where $q, p_{ij}$ are polynomials in $\mathbb{R}[t]$ such that $q$ has no real root, $\gcd(p_{11}, \dots, p_{44}, q) = 1$ and $\operatorname{deg}(q)\ge \operatorname{deg}(p_{ij})$ for all $1\le i,j \le 4$.
\end{definition}
By definition, a rational curve $\gamma$ on $\operatorname{SO}_4(\mathbb R)$ satisfies
\[
\gamma(t)\gamma(t)^{\top}=I_4,\qquad \det \gamma(t)=1,
\]
for all $t\in \mathbb{R}$. Following \cite{Liye2024rational}, we may define the degree of $\gamma$.

\begin{definition}\label{def:deg}
Let $\gamma(t) = \big( p_{ij}(t)/q(t) \big)_{1 \le i,j \le 4}$ be a rational curve on $\operatorname{SO}_4(\mathbb{R})$. The degree of $ \gamma $ is $ \deg(\gamma) := \deg(q)$, which is always an even non-negative integer.
\end{definition}
For rational curves $\gamma_1, \gamma_2$ on $\operatorname{SO}_4(\mathbb{R})$, it is straightforward to verify that $\deg(\gamma_1 \gamma_2) \le \deg(\gamma_1) + \deg(\gamma_2)$, and the inequality can strictly hold.

The set of all rational curves on $\operatorname{SO}_4(\mathbb{R})$ forms a group under multiplication, with identity $I_4$ and inverse $\gamma^{-1} = \gamma^\top$, denoted by $\operatorname{Rat}(4)$. Since $\operatorname{SO}_4(\mathbb{R})$ is compact, $\gamma(\infty) := \displaystyle \lim_{t \to \infty} \gamma(t)$ is well defined for every $\gamma \in \operatorname{Rat}(4)$. Let
\[
\operatorname{Rat}(4, I_4) := \{ \gamma \in \operatorname{Rat}(4) \mid \gamma(\infty) = I_4\},
\]
which is a subgroup. Every $\gamma \in \operatorname{Rat}(4)$ factors uniquely as $\gamma = A \gamma_0$ with $A = \gamma(\infty) \in \operatorname{SO}_4(\mathbb{R})$ and $\gamma_0 \in \operatorname{Rat}(4, I_4)$. Thus, it suffices to concentrate on curves in $\operatorname{Rat}(4, I_4)$. By definition, $\gamma(t) = \big( p_{ij}(t)/q(t) \big)_{1 \le i,j \le 4}$ in $\operatorname{Rat}(4, I_4)$ is characterized by the following conditions:
\begin{enumerate}
    \renewcommand{\labelenumi}{(\roman{enumi})}
    \item $\gamma(t) \in \operatorname{SO}_4(\mathbb{R})$ for all $t \in \mathbb{R}$.    
    \item $q$ has no real root.
    \item $\gcd(p_{11}, p_{12}\dots p_{44}, q) = 1$.
    \item $q$ and the diagonal elements $p_{11},\dots,p_{44}$ are monic.
    \item $\deg (q) = \deg (p_{ii}) > \deg (p_{ij})$ for all $i \neq j$.
\end{enumerate} 

In what follows, we shall always work with this parametrization. The goal of the subsequent sections is to analyze the structure of $\operatorname{Rat}(4, I_4)$ by decomposing such curves into factors of lower degree.

\subsection{Quaternions and isoclinic matrices}

We recall the algebra of quaternions
$\mathbb{H} = \{p_0 + p_1 \qi + p_2 \qj + p_3 \qk \mid p_0, p_1, p_2, p_3 \in \mathbb{R}\}$, where $\qi, \qj, \qk$ satisfy $\qi^2 = \qj^2 = \qk^2 = \qi\qj\qk = -1$. Given a quaternion $x = x_0 +  x_1 \qi +  x_2 \qj +  x_3 \qk$, we may identify it with the vector $(x_0, x_1, x_2, x_3)^{\top} \in \mathbb{R}^4$. This provides an isomorphism $\mathbb{H} \cong \mathbb{R}^4$ between vector spaces.

For a fixed quaternion $p = p_0 + p_1 \qi + p_2 \qj + p_3 \qk \in \mathbb{H}$, left multiplication defines a linear map 
$L_p \colon \mathbb{H} \to \mathbb{H}$, $x \mapsto p x$. Similarly, for a fixed quaternion $q = q_0 + q_1 \qi + q_2 \qj + q_3 \qk$, 
right multiplication defines $R_q \colon \mathbb{H} \to \mathbb{H}$, 
$x \mapsto x q$. Their representing matrices under the identification $\mathbb{H} \cong \mathbb{R}^4$ are given as follows:
\begin{equation}\label{eq:isoclinic} 
    \begin{split}
    L_p &= p_0 M_0 + p_1 M_1 + p_2 M_2 + p_3 M_3  = \begin{pmatrix}  
p_0 & -p_1 & -p_2 & -p_3\\
p_1 & \phantom{-}p_0 & -p_3 & \phantom{-}p_2\\
p_2 & \phantom{-}p_3 & \phantom{-}p_0 & -p_1 \\
p_3 & -p_2 & \phantom{-}p_1 & \phantom{-}p_0
\end{pmatrix}, \\
R_q &= q_0 M_0 + q_1 N_1 + q_2 N_2 + q_3 N_3 = \begin{pmatrix}  
\phantom{-}q_0 & \phantom{-}q_1 & \phantom{-}q_2 & \phantom{-}q_3\\
-q_1 & \phantom{-}q_0 & -q_3 & \phantom{-}q_2\\
-q_2 & \phantom{-}q_3 & \phantom{-}q_0 & -q_1 \\
-q_3 & -q_2 & \phantom{-}q_1 & \phantom{-}q_0
\end{pmatrix},
\end{split}
\end{equation}

where $M_0 = N_0 = I_4$ is the $4 \times 4$ identity matrix, and
\begin{equation}\label{eq:MN}
\begin{split}
M_1 &= \begin{pmatrix}  
0 & -1 & 0  & 0\\
1 & 0 & 0 & 0\\
0 & 0 & 0 & -1 \\
0 & 0 & 1 & 0
\end{pmatrix}, \quad M_2 = \begin{pmatrix}  
0 & 0 & -1  & 0\\
0 & 0 & 0 & 1\\
1 & 0 & 0 & 0 \\
0 & -1 & 0 & 0
\end{pmatrix}, \quad M_3 = \begin{pmatrix}  
0 & 0 & 0  & -1\\
0 & 0 & -1 & 0\\
0 & 1 & 0 & 0 \\
1 & 0 & 0 & 0
\end{pmatrix}, \\[5pt]
N_1 &= \begin{pmatrix}  
0 & 1 & 0  & 0\\
-1 & 0 & 0 & 0\\
0 & 0 & 0 & -1 \\
0 & 0 & 1 & 0
\end{pmatrix}, \quad N_2 = \begin{pmatrix}  
0 & 0 & 1  & 0\\
0 & 0 & 0 & 1\\
-1 & 0 & 0 & 0 \\
0 & -1 & 0 & 0
\end{pmatrix}, \quad N_3 = \begin{pmatrix}  
0 & 0 & 0  & 1\\
0 & 0 & -1 & 0\\
0 & 1 & 0 & 0 \\
-1 & 0 & 0 & 0
\end{pmatrix}.
\end{split}
\end{equation}
A direct computation yields the multiplication tables of these matrices, as shown in Table \ref{tab:multiplication}.

\begin{table}[h]
\centering
\renewcommand{\arraystretch}{1.2}
\setlength{\tabcolsep}{0.2pt} 
\begin{tabular}{|c|c|c|c|c|}
\hline
\(\;\;\times\;\;\) & \cellcolor{gray!30}\(\;\;\,M_0\;\;\,\) & \cellcolor{gray!30}\(M_1\) & \cellcolor{gray!30}\(M_2\) & \cellcolor{gray!30}\(M_3\) \\
\hline
\cellcolor{gray!30}\(M_0\) & \cellcolor{red!15}\(M_0\) & \cellcolor{blue!15}\(M_1\) & \cellcolor{yellow!15}\(M_2\) & \cellcolor{green!15}\(M_3\) \\
\hline
\cellcolor{gray!30}\(M_1\) & \cellcolor{blue!15}\(M_1\) & \cellcolor{red!15}\(-M_0\) & \cellcolor{green!15}\(M_3\) & \cellcolor{yellow!15}\(-M_2\) \\
\hline
\cellcolor{gray!30}\(M_2\) & \cellcolor{yellow!15}\(M_2\) & \cellcolor{green!15}\(-M_3\) & \cellcolor{red!15}\(-M_0\) & \cellcolor{blue!15}\(M_1\) \\
\hline
\cellcolor{gray!30}\(M_3\) & \cellcolor{green!15}\(M_3\) & \cellcolor{yellow!15}\(M_2\) & \cellcolor{blue!15}\(-M_1\) & \cellcolor{red!15}\(-M_0\) \\
\hline
\end{tabular} \qquad \qquad \quad
\begin{tabular}{|c|c|c|c|c|}
\hline
\(\;\;\times\;\;\) & \cellcolor{gray!30}\(\;\;\,N_0 \;\;\,\) & \cellcolor{gray!30}\(N_1\) & \cellcolor{gray!30}\(N_2\) & \cellcolor{gray!30}\(N_3\) \\
\hline
\cellcolor{gray!30}\(N_0\) & \cellcolor{red!15}\(N_0\) & \cellcolor{blue!15}\(N_1\) & \cellcolor{yellow!15}\(N_2\) & \cellcolor{green!15}\(N_3\) \\
\hline
\cellcolor{gray!30}\(N_1\) & \cellcolor{blue!15}\(N_1\) & \cellcolor{red!15}\(-N_0\) & \cellcolor{green!15}\(N_3\) & \cellcolor{yellow!15}\(-N_2\) \\
\hline
\cellcolor{gray!30}\(N_2\) & \cellcolor{yellow!15}\(N_2\) & \cellcolor{green!15}\(-N_3\) & \cellcolor{red!15}\(-N_0\) & \cellcolor{blue!15}\(N_1\) \\
\hline
\cellcolor{gray!30}\(N_3\) & \cellcolor{green!15}\(N_3\) & \cellcolor{yellow!15}\(N_2\) & \cellcolor{blue!15}\(-N_1\) & \cellcolor{red!15}\(-N_0\) \\
\hline
\end{tabular}
\caption{Multiplication tables of \(M_i\) and \(N_j\).}
\label{tab:multiplication}
\end{table}
From Table~\ref{tab:multiplication}, we observe that the matrices 
$M_i$ and $N_j$ 
satisfy the same multiplication rules as $\qi, \qj, \qk$. Furthermore, a direct computation shows that $M_i N_j = N_j M_i$ for all $1 \le i, j \le 3$, reflecting the fact that $M_i$ and $N_j$ generate two commuting Lie subalgebras of $\mathfrak{so}(4)$. These two subalgebras correspond to the standard factorization $\mathfrak{so}(4) \cong \mathfrak{so}(3) \oplus \mathfrak{so}(3)$, which is induced by the isomorphism $\operatorname{Spin}(4) \cong \operatorname{Spin}(3) \times \operatorname{Spin}(3)$~\cite{hall2013lie}.

\begin{definition}[Isoclinic matrices]
A matrix of the form $L_p$ with $p_0^2 + p_1^2 + p_2^2 + p_3^2 = 1$ 
is called a left isoclinic matrix. 
Similarly, a matrix of the form $R_q$ with $q_0^2 + q_1^2 + q_2^2 + q_3^2 = 1$ 
is called a right isoclinic matrix.
\end{definition}

Isoclinic matrices are orthogonal and have determinant one whenever $p_0^2+p_1^2+p_2^2+p_3^2=1$ or $q_0^2+q_1^2+q_2^2+q_3^2=1$. Hence both left and right isoclinic matrices lie in \(\operatorname{SO}_4(\mathbb{R})\). Geometrically, a left isoclinic matrix corresponds to a rotation with the same angle in two orthogonal planes, while a right isoclinic matrix corresponds to rotations with opposite angles. From the commutativity of $M_i$ and $N_j$, we naturally obtain the following well-known property.

\begin{proposition}[{\cite[Section~2]{Perez17cayley}}]
Left isoclinic matrices commute with right isoclinic matrices. 
\end{proposition}

\subsection{Cayley's factorization of 4D rotations}\label{subsec:Cayley}
The following is a classical result of Cayley.
\begin{theorem}[Cayley's factorization \cite{Cayley55,Weiner05quaternion}]\label{thm:Cayley}
For any $Q \in \operatorname{SO}_4(\mathbb{R})$ there exist a left isoclinic matrix $L$ and a right isoclinic matrix $R$, such that $Q = L R$.
\end{theorem}

We briefly recall the following explicit algorithm \cite{Perez17cayley} for this factorization. For each $0 \le i \le 3$, we define a linear operator $\mathscr{F}_i: \mathbb{R}^{4 \times 4} \to \mathbb{R}^{4 \times 4}$ as follows:
\begin{equation}\label{eq:Fi}
 \begin{split}
\mathscr{F}_0(X) &:= -\frac{1}{4} \left( -X + M_1 X M_1 + M_2 X M_2 + M_3 X M_3 \right), \\
\mathscr{F}_1(X) &:= -\frac{1}{4} \left( X M_1 + M_1 X + M_3 X M_2 - M_2 X M_3 \right), \\
\mathscr{F}_2(X) &:= -\frac{1}{4} \left( X M_2 + M_2 X + M_1 X M_3 - M_3 X M_1 \right), \\
\mathscr{F}_3(X) &:= -\frac{1}{4} \left( X M_3 + M_3 X + M_2 X M_1 - M_1 X M_2 \right),
\end{split}
\end{equation}
where $M_1, M_2, M_3$ are defined as \eqref{eq:MN}. Then for any $Q \in \operatorname{SO}_4(\mathbb{R})$, there is some $0 \le i \le 3$ such that $\det (\mathscr{F}_i(Q)) \ne 0$. Moreover, the factors $L$ and $R$ are given by
\[
R = \frac{\mathscr{F}_i(Q)}{\operatorname{det}\bigl(\mathscr{F}_i(Q)\bigr)^{1/4}},\quad L = Q R^\top = \frac{Q \mathscr{F}_i(Q)^\top}{\operatorname{det}\bigl(\mathscr{F}_i(Q)\bigr)^{1/4}}.
\]

\subsection{Quaternion polynomials}
For a given quaternion $q = q_0 + q_1\qi + q_2\qj + q_3\qk$, we denote the \emph{conjugate} of $q$ by $\Cj{q} := q_0 - q_1\qi - q_2\qj - q_3\qk$, and its \emph{norm} by $\QNorm{q} := q\Cj{q} = q_0^2 + q_1^2 + q_2^2 + q_3^2 \in \R$. We say a quaternion $q$ is purely imaginary if $q_0=0$, and in this case we have $\Cj{q}=-q$.

We now consider the algebra $\Hq[t]$ of polynomials in the indeterminate $t$ with coefficients in $\Hq$. Multiplication is defined by requiring that the indeterminate $t$ commutes with all quaternion coefficients; equivalently, $t$ is central over $\mathbb{H}$. This is only one possibility among many~\cite{Ore33theory}, but it is appropriate in kinematics because we think of $t$ as a real motion parameter and $\R$ is the center of $\mathbb{H}$. For $\mathcal{M} \in \mathbb{H}[t]$, we denote by $\Cj{\mathcal{M}}$ the polynomial obtained by conjugating coefficients of $\mathcal{M}$. The \emph{norm polynomial} of $\mathcal{M}$ is $\QNorm{\mathcal{M}} := \mathcal{M}\Cj{\mathcal{M}} \in \R[t]$.
\begin{definition}\label{def-primitive}
  A quaternion polynomial $\mathcal{M}=q_0 + q_1\qi + q_2\qj + q_3\qk 
  \in \Hq[t]$ with $q_0, q_1, q_2, q_3 \in \mathbb{R}[t]$ is called primitive if $\gcd(q_0, q_1,  q_2,  q_3)=1$. If $\deg (q_0) = \cdots = \deg (q_3) = 1$, then we say that $\mathcal{M}$ is a linear quaternion polynomial.
\end{definition} 

Every primitive monic quaternion polynomial admits a factorization into the product of linear factors.
\begin{theorem}[{\cite[Theorem 1]{Hegedus13}}] 
  \label{thm:alg_galg} 
  For a primitive monic quaternion polynomial $\mathcal{M}\in \Hq[t]$ of degree~$n$, there exist monic linear quaternion polynomials $\mathcal{L}_1,\dots,\mathcal{L}_n$ such that
    $\mathcal{M}=\mathcal{L}_1 \cdots \mathcal{L}_n$.
\end{theorem}
By \cite{Hegedus13}, there exists a Euclidean-type algorithm for the factorization in Theorem~\ref{thm:alg_galg}. For convenience, we reproduce it as Algorithm~\ref{alg:quaternion}.
\begin{algorithm}[H]
\caption{\texttt{GFactor}: Factorization algorithm for a primitive monic quaternion polynomial}
\label{alg:quaternion}
\begin{algorithmic}[1]
\Require A primitive monic quaternion polynomial $\mathcal{M}$ of degree $n$.
\Ensure A list $[\mathcal{L}_1,\ldots,\mathcal{L}_{n}]$ of monic linear quaternion polynomials such that $\mathcal{M}=\mathcal{L}_1 \cdots \mathcal{L}_n$.

\If{$\deg (\mathcal{M}) = 0$}
      \State \Return $[\,]$ \Comment{Empty list.}
\EndIf

\State $\rho(t) \adef$ a monic quadratic real factor of the norm polynomial $\mathcal{M}\Cj{\mathcal{M}} \in \R[t]$.
    \label{galg:pick-quadratic-factor}
    \State Compute $h$ from the remainder of $\mathcal{M}=\mathcal{Q}\rho+\mathcal{R}=\mathcal{Q}(t-\Cj{h})(t-h)+r_1(t-h)$ \Comment{$\mathcal{M}(h)=\rho(h)=0$}
    \label{galg:h}
    \State $\mathcal{M} \adef \frac{\mathcal{M}(t-\Cj{h})}{\rho(t)}$
    \State \Return $\mathtt{GFactor}(\mathcal{M}) \mathbin{\Vert} [t-h]$ \Comment{$\Vert$ denotes list concatenation.}
\end{algorithmic}
\end{algorithm}

\section{Factorization into quadratic rational curves}\label{sec:quadratic-rational-decomp}

\subsection{Isoclinic rational curves}

In this section, we introduce two distinguished classes of rational curves in $\operatorname{SO}_4(\mathbb R)$, namely left-isoclinic and right-isoclinic rational curves. We establish a one-to-one correspondence between these curves and quaternion polynomials. This correspondence reduces the decomposition problem for rational curves in $\operatorname{SO}_4(\mathbb R)$ to the factorization problem for quaternion polynomials. To this end, we first introduce the notion of polynomial matrices of isoclinic form.

\begin{definition}\label{def:isomat}
Let $\operatorname{M}_4(\mathbb{R}[t])$ be the ring of $4 \times 4$ matrices over $\mathbb{R}[t]$. An element $M \in \operatorname{M}_4(\mathbb{R}[t])$ is called a polynomial matrix. We say that $M$ is of left isoclinic form if there exist $\ell_0, \ell_1, \ell_2, \ell_3 \in \mathbb{R}[t]$ such that  
\[
M = \begin{pmatrix}  
\ell_0 & -\ell_1 & -\ell_2 & -\ell_3\\
\ell_1 & \phantom{-}\ell_0 & -\ell_3 & \phantom{-}\ell_2\\
\ell_2 & \phantom{-}\ell_3 & \phantom{-}\ell_0 & -\ell_1 \\
\ell_3 & -\ell_2 & \phantom{-}\ell_1 & \phantom{-}\ell_0
\end{pmatrix} \eqqcolon [\ell_0, \ell_1, \ell_2, \ell_3]_L. 
\]
The set consisting of all polynomial matrices of left isoclinic form is denoted by $\operatorname{Q}_{L}(\mathbb{R}[t])$. 

Similarly, we say that $M$ is of right isoclinic form if there exist $r_0, r_1, r_2, r_3 \in \mathbb{R}[t]$ such that  
\[
M = \begin{pmatrix}  
\phantom{-}r_0 & \phantom{-}r_1 & \phantom{-}r_2 & \phantom{-}r_3\\
-r_1 & \phantom{-}r_0& -r_3 & \phantom{-}r_2\\
-r_2 & \phantom{-}r_3 & \phantom{-}r_0 & -r_1\\
-r_3 & -r_2 & \phantom{-}r_1 & \phantom{-}r_0
\end{pmatrix} \eqqcolon [r_0, r_1, r_2, r_3]_R,
\]
and $\operatorname{Q}_{R}(\mathbb{R}[t])$ denotes the set of all polynomial matrices of right isoclinic form.
\end{definition}

With this notation, we now define the left and right isoclinic rational curves. We recall from Definition~\ref{def:ratcurve} that a rational curve $\gamma$ on $\operatorname{SO}_4(\mathbb{R})$ can be written as $\gamma = P/q$ for some $P \in \operatorname{M}_4(\mathbb{R}[t])$ and $q \in \mathbb{R}[t]$.

\begin{definition}
A rational curve $\gamma = P/q$ on $\operatorname{SO}_4(\mathbb{R})$ is left (resp. right) isoclinic if $P$ is of left (resp. right) isoclinic form $[\ell_0,\ell_1,\ell_2,\ell_3]_L$ (resp. $[r_0,r_1,r_2,r_3]_R$) such that $\ell_0^2 + \cdots + \ell_3^2 = q^2$ (resp. $r_0^2 + \cdots + r_3^2 = q^2$).
\end{definition}

Let $\alpha = A/q_1$ (resp. $\beta = B/q_2$) be a left (resp. right) isoclinic rational curve on $\operatorname{SO}_4(\mathbb{R})$ with $A = [\ell_0, \ell_1, \ell_2, \ell_3]_L$ and $B = [r_0, r_1, r_2, r_3]_R$ for some $\ell_i, r_i \in 
\mathbb{R}[t]$, $0 \le i \le 3$. Then $A$, $B$ can be written as
\begin{equation}\label{eq:AB}
A = \ell_0 M_0 + \ell_1 M_1 + \ell_2 M_2 + \ell_3 M_3,\qquad B =  r_0 N_0 + r_1 N_1 + r_2 N_2 + r_3 N_3. 
\end{equation}
where $M_i, N_i$ are defined as \eqref{eq:MN}, $0 \le i, j \le 3$. The commutativity of $M_i$ and $N_j$ yields the following lemma.

\begin{lemma}\label{lem:commu}
Let $\alpha$ (resp. $\beta$) be a left (resp. right) isoclinic rational curve on $\operatorname{SO}_4(\mathbb{R})$. If $\deg (\alpha) = 2d_1$ and $\deg (\beta) = 2d_2$, then $\alpha\beta = \beta\alpha$ and $\deg (\alpha \beta) = 2(d_1+d_2)$.
\end{lemma}
\begin{proof}
    Let $\alpha = A/q_1$, $\beta = B/q_2$ where $\deg(q_1) = 2d_1$, $\deg(q_2) = 2d_2$ and $A,B$ are parametrized as in \eqref{eq:AB}. Then a direct calculation implies 
    \[
    \alpha \beta = \frac{1}{q_1 q_2} \sum_{i, j=0}^3 \ell_i r_jM_iN_j = \frac{1}{q_1 q_2} \sum_{i, j=0}^3 r_j \ell_i N_jM_i = \beta \alpha.
    \]

    Let $P \coloneqq AB = \bigl( p_{mn} \bigr)_{1 \le m,n \le 4} \in \operatorname{M}_4(\mathbb{R}[t])$.
    To prove $\deg (\alpha \beta) = 2(d_1 + d_2)$, it suffices to show that $\gcd (p_{11}, p_{12}\dots p_{44}) = 1$. To this end, we write $P$ explicily as 
    \[
    P = \begin{pmatrix}
        u_{00}-u_{11}-u_{22}-u_{33} & -u_{01}-u_{10}+u_{23}-u_{32} & -u_{02}-u_{13}-u_{20}+u_{31} & -u_{03}+u_{12}-u_{21}-u_{30} \\
        u_{01}+u_{10}+u_{23}-u_{32} & u_{00}-u_{11}+u_{22}+u_{33} & u_{03}-u_{12}-u_{21}-u_{30} & -u_{02}-u_{13}+u_{20}-u_{31} \\
        u_{02}-u_{13}+u_{20}+u_{31} & -u_{03}-u_{12}-u_{21}+u_{30} & u_{00}+u_{11}-u_{22}+u_{33} & u_{01}-u_{10}-u_{23}-u_{32} \\
        u_{03}+u_{12}-u_{21}+u_{30} & u_{02}-u_{13}-u_{20}-u_{31} & -u_{01}+u_{10}-u_{23}-u_{32} & u_{00}+u_{11}+u_{22}-u_{33}
    \end{pmatrix},
    \]
    where $u_{ij} \coloneqq l_ir_j$, $0 \le i,j \le 3$. We observe that each $u_{ij}$ is a linear combination of $p_{mn}$'s. In particular, we have 
    \begin{equation}\label{eq:lrmultip_prop}
        \begin{aligned}
            &p_{11} + p_{22} + p_{33} + p_{44} = 4 u_{00} =4 \,l_0 r_0, \qquad p_{33} + p_{44} - p_{11} + p_{22} = 4 u_{11} = 4\, l_1r_1,\\
            &p_{22} + p_{44} - p_{11} - p_{33} = 4 u_{22} = 4 \,l_2r_2,\qquad p_{22} + p_{33} - p_{11} - p_{44} = 4 u_{33} = 4 \,l_3r_3.
        \end{aligned}
    \end{equation} 
    Denote $c \coloneqq \gcd(p_{11}, p_{12}\dots p_{44})$. Then we must have $c \mid l_i r_j$ for all \(0 \le i, j \le 3\). Since $\alpha$ and $\beta$ are rational curves, we deduce that $\gcd(\ell_0, \ell_1, \ell_2, \ell_3) = \gcd(r_0, r_1, r_2, r_3) = 1$. If $c \ne 1$ and $\xi$ is an irreducible real factor of $c$, then there exist $0 \le i_0, j_0 \le 3$ such that \(\xi \nmid \ell_{i_0}\) and \(\xi \nmid r_{j_0}\).    
    This implies \(\xi \nmid l_{i_0} r_{j_0}\), contradicting the fact that \(c \mid l_{i_0} r_{j_0}\).
\end{proof}

Next, we consider two $\mathbb{R}$-linear maps
\begin{equation}\label{phi_LR}
    \begin{split}
      &\varphi_L : \operatorname{Q}_{L}(\mathbb{R}[t]) \to \mathbb{H}[t],\quad \varphi_L(M) = \ell_0 + \ell_1\qi + \ell_2\qj + \ell_3\qk,  \\
      &\varphi_R : \operatorname{Q}_{R}(\mathbb{R}[t]) \to \mathbb{H}[t],\quad \varphi_R(N) =  r_0 + r_1\qi + r_2\qj + r_3\qk, 
    \end{split}
\end{equation}
where $M = [\ell_0, \ell_1, \ell_2, \ell_3]_L \in \operatorname{Q}_{L}(\mathbb{R}[t])$ and $N = [r_0, r_1, r_2, r_3]_R \in \operatorname{Q}_{R}(\mathbb{R}[t])$. The following lemma establishes an equivalence between isoclinic rational curves and quaternion polynomials.

\begin{lemma}\label{lem:keep_multip}
The maps $\varphi_L$ and $\varphi_R$ defined in~\eqref{phi_LR} are isomorphisms of $\mathbb{R}$-algebras.
\end{lemma}
\begin{proof}
We define the following $\mathbb{R}$-linear maps:
\begin{equation}\label{psi_LR}
    \begin{split}
        &\psi_L: \mathbb{H}[t] \to \operatorname{Q}_{L}(\mathbb{R}[t]),\quad  \psi_L(\mathcal{P})  = [p_0, p_1, p_2, p_3]_L, \\
        &\psi_R: \mathbb{H}[t] \to \operatorname{Q}_{R}(\mathbb{R}[t]),\quad \psi_R(\mathcal{Q}) = [q_0,q_1,q_2,q_3]_R,  
    \end{split}
\end{equation}
where $\mathcal{P} = p_0 + p_1\qi + p_2\qj + p_3\qk \in \mathbb{H}[t]$ and $\mathcal{Q} = q_0 + q_1\qi + q_2\qj + q_3\qk \in \mathbb{H}[t]$. It is straightforward to verify that $\varphi_L \circ \psi_L = \operatorname{id}_{\mathbb{H}[t]}$ 
and $\psi_L \circ \varphi_L = \operatorname{id}_{\operatorname{Q}_{L}(\mathbb{R}[t])}$. By Table~\ref{tab:multiplication}, one may easily check that $\varphi_L(L_1 L_2) = \varphi_L(L_1) \varphi_L(L_2)$. The proof for $\varphi_R$ is identical and is therefore omitted.
\end{proof}

\subsection{Quadratic factorization}

This subsection is devoted to proving Theorem~\ref{thm:decomp}, which states that every rational curve on $\operatorname{SO}_4(\mathbb{R})$ can be decomposed into a product of quadratic rational curves. To achieve this goal, we first decompose rational curves as a product of isoclinic curves.

Let $\gamma$ be a rational curve on $\operatorname{SO}_4(\mathbb{R})$. We recall from Subsection~\ref{subsec:curveSO4} that $\gamma = \gamma(\infty) \gamma_0$ for some $\gamma_0 \in \operatorname{Rat}(4,I_4)$. Here $\operatorname{Rat}(4,I_4)$ consists of rational curves whose limit at $\infty$ is $I_4$. Thus, it is sufficient to assume $\gamma \in \operatorname{Rat}(4,I_4)$. From the four operators $\mathscr{F}_i$ defined in \eqref{eq:Fi}, we choose an operator $\mathscr{F}_0$ that acts non-trivially on $\gamma$, and denote $\mathscr{F} \coloneqq \mathscr{F}_0$.

\begin{lemma}\label{lem:rightfunctor}
Let $\gamma$ and $\mathscr{F}$ be as above. Then $\mathscr{F}(\gamma) = F/q$, where $F = [f_0, f_1, f_2, f_3]_R \ne 0$ for some $f_0,\dots, f_3 \in \mathbb{R}[t]$. In particular, we have $\mathscr{F}(\gamma)^\top \mathscr{F}(\gamma) = \mathscr{F}(\gamma) \mathscr{F}(\gamma)^\top = \sum_{i=0}^3 ({f_i}/q)^2 I_4$.
\end{lemma}

\begin{proof}
We write $\gamma = P/q$ for some $P = (p_{ij})_{1 \le i,j \le 4} \in \mathrm{M}_4(\mathbb{R}[t])$ and $q \in \mathbb{R}[t]$ as in Subsection~\ref{subsec:curveSO4}. Denote 
\begin{align*}
    f_0& \coloneqq (p_{11} + p_{22} + p_{33} + p_{44})/4,\qquad
f_1 \coloneqq (p_{21} - p_{12} - p_{43} + p_{34})/4,\\
f_2&\coloneqq (p_{31} + p_{42} - p_{13} - p_{24})/4,\qquad
f_3\coloneqq (p_{41} - p_{32} + p_{23} - p_{14})/4.
\end{align*}
A direct calculation yields $\mathscr{F}(\gamma) = F/q$. To check $\mathscr{F}(\gamma) \ne 0$, we notice that by assumption, it holds that $\gamma(\infty) = I_4$. This implies that $P = q_0 I_4 + Q$, where $q_0$ is the leading homogeneous part of $q$ and each element of $Q \in \mathrm{M}_4(\mathbb{R}[t])$ has degree strictly less than $\deg (q)$. Since $\mathscr{F}$ is a linear operator, we conclude that $\mathscr{F}(\gamma) = I_4 + \mathscr{F}(Q)/q \ne 0$. 
\end{proof}
It is worth remarking that $\mathscr{F}(\gamma)$ is a rational curve on $\operatorname{SO}_4(\mathbb{R})$ only if $\sum_{i=0}^3 (f_i/q)^2 = 1$. 

\begin{example}
    Let $\gamma = P/(t^2+1)^2$ be the rational curve on $\operatorname{SO}_4(\mathbb{R})$ of degree $4$ with poles $\pm \qi$, where
    $$P(t) =
    \begin{pmatrix}
t^{4} - t^{2} - 2t & -2t^{3} - t^{2} + 1 & -t^{3} + 2t^{2} + t & -t^{3} - 2t^{2} + t \\
2t^{3} + t^{2} - 1 & t^{4} - t^{2} - 2t & -t^{3} - 2t^{2} + t & t^{3} - 2t^{2} - t \\
t^{3} + t & t^{3} + t & t^{4} + t^{2} & -t^{2} - 1 \\
t^{3} + t & -t^{3} - t & t^{2} + 1 & t^{4} + t^{2}
\end{pmatrix}.$$
Then we have
$$\mathscr{F}(\gamma(t)) = \frac{t^3-1}{(t^2+1)^2}
\begin{pmatrix}
t & -1 & 0 & 0 \\
1 & t & 0 & 0 \\
0 & 0 & t & 1 \\
0 & 0 & -1 & t
\end{pmatrix}.$$
This matrix is not a rational curve on $\operatorname{SO}_4(\mathbb{R})$, since its columns have squared norm $(t^3-1)/(t^2+1)$ rather than $1$.
\end{example}

Next, we consider the operator
\begin{equation}\label{eq:functor_G}
\mathscr{G}: \mathbb{R}^{4\times 4} \to \mathbb{R}^{4\times 4},\qquad \mathscr{G}(X) \coloneqq X \mathscr{F}(X)^\top = \frac{1}{4}(X X^\top - X  M_1  X^\top M_1 - X  M_2  X^\top M_2 - X M_3  X^\top M_3).
\end{equation}

\begin{lemma}\label{lem:leftfunctor}
For any $\gamma \in \operatorname{Rat}(4,I_4)$, we have $\mathscr{G}(\gamma) = G/q^2$ for some nonzero $G \in \mathrm{M}_4(\mathbb{R}[t])$ of left isoclinic form.
\end{lemma}
\begin{proof}
By Lemma~\ref{lem:rightfunctor}, we have $\mathscr{F}(\gamma)\ne 0$. Since $\gamma$ is a rational curve on $\operatorname{SO}_4(\mathbb{R})$, we deduce that $\mathscr{G}(\gamma) = \gamma \mathscr{F}(\gamma)^\top \ne 0$. We write $\gamma = P/q$ for some $P \in \mathrm{M}_4(\mathbb{R}[t])$ and $q \in \mathbb{R}[t]$ as in Subsection~\ref{subsec:curveSO4}. According to Lemma~\ref{lem:rightfunctor}, there exists some $F \in \mathrm{M}_4(\mathbb{R}[t])$ of right isoclinic form such that $\mathscr{F}(\gamma) = F/q$. By definition, $\mathscr{G}(\gamma) = P F^\top/q^{2}$. Denote $G \coloneqq P F^\top \in \mathrm{M}_4(\mathbb{R}[t])$. We claim that $G$ is of left isoclinic form. 

Let $t_0$ be a fixed real number. Then $\gamma(t_0)$ is a fixed matrix in $\operatorname{SO}_4(\mathbb{R})$. By the discussion in Subsection~\ref{subsec:Cayley}, $\gamma(t_0)$ admits a factorization $\gamma(t_0) = L R$ for some left isoclinic matrix $L$ and right isoclinic matrix $R$. Suppose that $a$ is the real number such that $a R = \mathscr{F}(\gamma(t_0))$. Hence $\mathscr{G}(\gamma(t_0)) = LR (aR)^\top = a L$ and $G(t_0) = a q(t_0)^2 L$. By Definition~\ref{def:isomat}, the subspace $\operatorname{Q}_{L}(\mathbb{R}[t])$ consisting of polynomial matrices of left isoclinic form, is defined by some linear equations $h_1,\dots, h_s \in \mathbb{Z}[x_{11},\dots, x_{44}]$. Since $h_1 (G(t_0)) = \cdots = h_s (G(t_0)) = 0$ for any fixed $t_0 \in \mathbb{R}$ and coefficients of $h_i$'s are integers, we conclude that $h_1(G) = \cdots = h_s(G) = 0$, thereby $G \in \operatorname{Q}_{L}(\mathbb{R}[t])$.
\end{proof}

Given $\gamma = P/q \in \operatorname{Rat}(4,I_4)$, Lemmas~\ref{lem:rightfunctor} and \ref{lem:leftfunctor} imply $\mathscr{F}(\gamma) = F/q$ and $\mathscr{G}(\gamma) = G/q^2$, where $F = [f_0, f_1, f_2, f_3]_R$ and $G = [g_0, g_1, g_2, g_3]_L$ for some $f_i, g_i \in \mathbb{R}[t]$, $0 \le i \le 3$. Here $f_0, g_0$ are monic by definition, and $\deg (f_0) > \deg (f_i)$ for $1 \le i \le 3$. Note that both $F$ and $G$ are nonzero. Hence not all of the $f_i$'s are zero, and similarly, not all of the $g_i$'s are zero. Denote by $\delta_f \coloneqq \gcd(f_0, f_1, f_2, f_3)$ and $\delta_g \coloneqq \gcd(g_0, g_1, g_2, g_3)$ the monic greatest common divisors. Then both $\delta_f$ and $\delta_g$ are nonzero, and we define
\begin{equation}\label{eq:def_LR}
    R \coloneqq \frac{F}{\delta_f}, \qquad L  \coloneqq \frac{G}{\delta_g}.
\end{equation}
By definition, $\delta_f$ and $\delta_g$ are the greatest common divisors of the four coefficient polynomials of $F$ and $G$. Dividing by them removes all nonconstant common factors. Since this division is by scalar real polynomials, it preserves the right and left isoclinic linear relations among the matrix entries; hence $R$ and $L$ have primitive coefficient tuples and retain their right and left isoclinic forms. Moreover, the diagonal elements of both $L$ and $R$ are monic, and each has degree strictly larger than that of any non-diagonal element.

\begin{lemma}\label{lem:LRrelation}
Let $L$ and $R$ be defined as in~\eqref{eq:def_LR}. Then $\gamma = L R/q$.
\end{lemma}
\begin{proof}
By the construction of $L$ and $R$, we have
\[
\mathscr{F}(\gamma)  = \frac{\delta_f R}{q}, \qquad \mathscr{G}(\gamma) =\frac{\delta_g L}{q^{2}}.
\]
We write $L = [\ell_0, \ell_1, \ell_2, \ell_3]_L$ and $R = [r_0, r_1, r_2, r_3]_R$. The construction of $L$ and $R$ gives $\operatorname{gcd}(\ell_0, \ell_1, \ell_2, \ell_3) = \operatorname{gcd}(r_0, r_1, r_2, r_3) = 1$, since the common scalar factors are removed by $\delta_g$ and $\delta_f$, respectively. Denote $m \coloneqq {\delta_f}^2\sum_{i=0}^{3}{r_i}^2$. Then 
\begin{equation*} 
\frac{\delta_f \delta_g LR}{q^3}= \mathscr{G}(\gamma) \mathscr{F}(\gamma) =  \frac{PF^\top 
    F}{q^3} =  \frac{mP}{q^{2}} \gamma.
\end{equation*}
Here the last equation follows from Lemma~\ref{lem:rightfunctor}. This implies
\begin{equation}\label{eq:LRP}
    \delta_f \delta_g LR  = m P.
\end{equation}
We claim that $m =\delta_f\delta_g$, from which we immediately obtain $\gamma = P/q = LR/q$.

To prove the claim, we set 
\[
\delta \coloneqq \delta_f\delta_g, \qquad Q \coloneqq LR = \bigl(q_{ij}\bigr)_{1 \le i,j \le 4} ,\qquad P =   \bigl(p_{ij}\bigr)_{1 \le i,j \le 4}.
\]
Since $\gcd(p_{11},p_{12} \dots p_{44}) =1$, equation~\eqref{eq:LRP} implies $\delta \mid m$. For the reverse divisibility, we show that elements of $Q$ have no non-constant common divisor, so that \eqref{eq:LRP} implies $m \mid \delta$. We proceed by contradiction. Suppose that there is an irreducible polynomial $\xi \in \mathbb{R}[t]$ dividing every element of $Q$. By expanding the product $Q = LR$ and taking linear combinations of the elements of $Q$, one obtains all products $\ell_ir_j$, $0\le i,j\le 3$, up to nonzero real scalar factors. Hence $\xi$ divides $\ell_ir_j$ for each $0\le i,j\le 3$. Since $\gcd(\ell_0,\ell_1,\ell_2,\ell_3)= 1$, there is an index $i_0$ such that $\xi \nmid l_{i_0}$. Then $\xi\mid r_j$ for every $0 \le j \le 3$, contradicting $\gcd(r_0,r_1,r_2,r_3)=1$.
\end{proof}

We are now ready to prove the first factorization theorem. 

\begin{theorem}\label{thm:decomp}
Every $\gamma \in \operatorname{Rat}(4,I_4)$ of degree $2d$ is a product of $d$ quadratic rational curves in $\operatorname{Rat}(4,I_4)$.
\end{theorem}

\begin{proof}
We write $\gamma = P/q$ for some $P = (p_{ij})_{1 \le i,j \le 4} \in \mathrm{M}_4(\mathbb{R}[t])$ and $q \in \mathbb{R}[t]$ as in Subsection~\ref{subsec:curveSO4}. By Lemma~\ref{lem:LRrelation}, we have $\gamma = LR/q$, where $L \in \operatorname{Q}_L(\mathbb{R}[t])$, $R \in \operatorname{Q}_R(\mathbb{R}[t])$ are defined as in \eqref{eq:def_LR}. Since $\gamma\gamma^\top = I_4$, the proof of Lemma~\ref{lem:commu} implies 
$(LL^\top)(RR^\top) = (LR)(LR)^{\top} = q^2 I_4$. Note that $q$ is monic and has no real roots. We may write $q=q_1 \cdots q_d$ for some monic irreducible quadratic polynomials $q_1,\dots, q_d$ in $\mathbb R[t]$, counted with multiplicity. Then $(LL^\top)(RR^\top) = q_1^2 \cdots q_d^2 I_4$. We observe that $LL^\top = f I_4$ and $RR^\top = g I_4$ for some $f,g\in \mathbb{R}[t]$. After reordering the $q_k$'s, we may write
\[
f = q_1 \cdots q_s \, q_{s+1}^2 \cdots q_{s+a}^2, \qquad
g = q_1 \cdots q_s \, q_{s+a+1}^2 \cdots q_d^2,
\]
where $s, a \ge 0$ and $s + a \le d$. Here $s$ counts the irreducible factors shared by $f$ and $g$, $a$ counts those
appearing only in $f$, and the remaining $d - s - a$ appear only in $g$.

Let $\varphi_L, \varphi_R$ and $\psi_L, \psi_R$ be the maps defined by~\eqref{phi_LR} and~\eqref{psi_LR}, and set
\[
\mathcal{L} \coloneqq \varphi_L(L) \in \mathbb{H}[t], \qquad \mathcal{R} \coloneqq \varphi_R(R) \in \mathbb{H}[t].
\]
Note that the maps $\varphi_L$ and $\varphi_R$ record the four coefficient polynomials of $L$ and $R$, respectively. Since these coefficient polynomials have no nonconstant common divisor, the quaternion polynomials $\mathcal{L}$ and $\mathcal{R}$ are primitive. Moreover, both $\mathcal{L}$ and $\mathcal{R}$ are monic. Applying Theorem~\ref{thm:alg_galg} to $\mathcal{L}$ and $\mathcal{R}$ gives decompositions of the form
\[
\mathcal{L} = (t - h_{L,1}) \cdots (t - h_{L,s + 2a}), \qquad
\mathcal{R} = (t - h_{R,1}) \cdots (t - h_{R,2d - s - 2a}),
\]
for some $h_{L,1},\dots, h_{L,s + 2a},h_{R,1},\dots, h_{R,2d - s - 2a} \in \mathbb{H}$ and
\begin{equation*}
    \begin{split}
        \{\, \QNorm{t - h_{L,1}}, \dots, \QNorm{t - h_{L,s + 2a}} \,\}
&= \{ q_1, \dots, q_s, q_{s+1}, q_{s+1}, \dots, q_{s+a}, q_{s+a} \},\\
\{\, \QNorm{t - h_{R,1}}, \dots, \QNorm{t - h_{R,2d - s - 2a}} \,\}
&= \{ q_1, \dots, q_s, q_{s+a+1}, q_{s+a+1}, \dots, q_{d}, q_{d} \}.
    \end{split}
\end{equation*}

Let $L_i \coloneqq \psi_L(t - h_{L,i})$ for $1 \le i \le s + 2a$ and 
$R_j \coloneqq \psi_R(t - h_{R,j})$ for $1 \le j \le 2d - s - 2a$. 
By Lemma~\ref{lem:keep_multip}, we have
\[
L = L_1 \cdots L_{s + 2a}, \qquad 
R = R_1 \cdots R_{2d - s - 2a}.
\]
Thus we obtain
\[
\gamma = \frac{ (L_1 R_1) \cdots (L_s R_s) \, (L_{s+1}L_{s+2}) \cdots (L_{s+2a-1}L_{s + 2a}) \, (R_{s+1}R_{s+2})  \cdots (R_{2d-s-2a -1}R_{2d - s - 2a})}{q}.
\]
The reordering here only moves right-isoclinic factors past left-isoclinic factors, which is legitimate by Lemma~\ref{lem:commu}. For each $1 \le k \le d$, we define 
\[
P_k(t) \coloneqq
\begin{cases}
L_k R_k & \text{if~} 1 \le k \le s,\\
L_{s+2i-1} L_{s+2i} & \text{if~}k = s+i \text{~and~} 1 \le i \le a,\\
R_{s+2j-1} R_{s+2j} & \text{if~}k = s+a+j \text{~and~} 1 \le j \le d - s - a.
\end{cases}
\]
By definition, $P_k P_k^\top = q_k^2 I_4$. Thus, $\gamma_k \coloneqq P_k / q_k$ satisfies $\gamma_k \gamma_k^\top = I_4$ and it is a quadratic rational curve on $\operatorname{SO}_4(\mathbb{R})$. Therefore $\gamma = \gamma_1 \cdots \gamma_d$. This completes the proof.
\end{proof}

\begin{remark}
The factorization in Theorem~\ref{thm:decomp} is not unique. Any pairing of the $2d$ linear factors that puts two factors
of the same norm together yields a valid factorization, provided the pairing can be realized
by swaps between $L_i$'s and $R_j$'s, without exchanging two $L_i$'s or two $R_j$'s past each other.
\end{remark}

The proof of Theorem~\ref{thm:decomp} is constructive. It not only establishes the existence of a factorization for rational curves on $\operatorname{SO}_4(\mathbb{R})$ but also provides an explicit algorithm for obtaining such a factorization.

\begin{algorithm}
\caption{Factorization algorithm for rational curves on $\operatorname{SO}_4(\mathbb{R})$ into quadratic rational curves}
\label{alg:decomp}
\begin{algorithmic}[1]
\Require A rational curve $\gamma = P/q \in \operatorname{Rat}(4,I_4)$ of degree $2d$.
\Ensure A list $[\gamma_1,\ldots,\gamma_{d}]$ of quadratic rational curves in $\operatorname{Rat}(4,I_4)$ such that $\gamma = \gamma_1 \cdots \gamma_{d}$.

\State Compute $\mathscr{F}(\gamma)$ and $\mathscr{G}(\gamma) = \gamma \mathscr{F}(\gamma)^\top$.

\State Extract factors $L$ and $R$ from $\mathscr{F}(\gamma) = (\delta_f/q) R$ and $\mathscr{G}(\gamma) = (\delta_g/q^{2}) L$.

\State Convert to quaternion polynomials $\mathcal{L} = \varphi_L(L)$, $\mathcal{R} = \varphi_R(R)$.

\State Factor $\mathcal{L}(t)$, $\mathcal{R}(t)$ into $ \prod (t-h_{L,i})$, $\prod (t-h_{R,j})$ such that equal norm polynomials are paired (Algorithm~\ref{alg:quaternion}).

\State Map the linear factors back via $\psi_L$ and $\psi_R$ to obtain linear matrix polynomials $L_i$ and $R_j$.

\State Multiply $L_i$, $R_j$ pairs with equal norm polynomial $q_k$ to obtain quadratic $P_k$, and set $\gamma_k \coloneqq P_k/q_k$.

\State \Return $[\gamma_1,\ldots,\gamma_{d}]$.
\end{algorithmic}
\end{algorithm}

\begin{corollary}
Algorithm~\ref{alg:decomp} factorizes any $\gamma \in \operatorname{Rat}(4,I_4)$ of degree $2d$ into exactly $d$ quadratic rational curves in $\operatorname{Rat}(4,I_4)$.
\end{corollary}

\begin{remark}
If $\gamma$ is a left isoclinic rational curve, then $\mathscr{F}(\gamma) = (\delta_f/q) I_4$ for some $\delta_f \in \mathbb{R}[t]$, and if $\gamma$ is right isoclinic, then $\mathscr{F}(\gamma) = \gamma$. In either case, Algorithm~\ref{alg:decomp} can be simplified by starting directly from Step~3.
\end{remark}

\section{Factorization into planar rotation curves}\label{sec:planar-rotation-decomp}

In the previous section, we show that any rational curve on $\operatorname{SO}_4(\mathbb{R})$ admits a factorization into a product of quadratic rational curves. This is connected to \emph{Kempe factorizations}, where a motion polynomial (a
rational curve on the Study quadric) is decomposed into linear factors. The  purpose of this section is to refine the factorization in Theorem~\ref{thm:decomp} by planar rotations defined as follows.

\begin{definition}\label{def:planar_rotation}
A quadratic rational curve $\gamma$ on $\operatorname{SO}_4(\mathbb{R})$ is called a planar rotation curve if there is a $2$-dimensional subspace $\mathbb{V} \subseteq \mathbb{R}^4$ such that $\gamma(t)|_{\mathbb{V}}$ is the identity map on $\mathbb{V}$ for any $t \in \mathbb{R}$. 
\end{definition}

Planar rotation curves are the geometrically primitive simple $4$-dimensional rotations. Indeed, we prove below that each quadratic rational curve in $\operatorname{Rat}(4,I_4)$ is a product of at most two planar rotation curves. Before proceeding, we recall that $\mathbb{H}_1$ is the group consisting of all unit quaternions, and we record an elementary fact that will be needed in the sequel.

\begin{lemma}\label{lem:elemfact}
If $h \in \mathbb{H}_1$ satisfies $h^2 = -1$, then there exists $a \in \mathbb{H}_1$ such that $a h\Cj{a}=\qi$.
\end{lemma}

\begin{proof}
Write $h=a_0+v$, where $a_0\in\mathbb R$ and $v$ is purely imaginary. Then $h^2=(a_0^2-\|v\|^2)+2a_0v$. Since $h^2=-1$, the imaginary part gives $2a_0v=0$. As $h$ is a unit quaternion and $h^2=-1$, we have $v\ne0$, hence $a_0=0$. Thus $h$ is purely imaginary, and we may write $h =x\qi+y\qj+z\qk$ for some $x,y,z\in \mathbb{R}$ with $x^2+y^2+z^2=1$. If $x =-1$, then $a=\qk$ is the desired unit quaternion. Otherwise, we take $a\coloneqq (h+\qi)/\sqrt{2+2x}$. A direct calculation yields $a h\Cj{a}=\qi$ and this completes the proof.
\end{proof}
In fact, Lemma~\ref{lem:elemfact} shows that the conjugation action of $\mathbb{H}_1$ on the space of purely imaginary unit quaternions is transitive. 

A closer investigation of the proof of Theorem~\ref{thm:decomp} reveals exactly three types of quadratic factors appearing in the factorization of a curve in $\operatorname{Rat}(4,I_4)$, which arise from three distinct pairings: left-right pairing, left-left pairing, and right-right pairing. Furthermore, given a quadratic rational curve $\alpha= A/q\in \operatorname{Rat}(4,I_4)$, we may, up to a linear reparametrization, assume that $q=t^2+1$. Since such a reparametrization leaves the factorization problem unchanged, it suffices to restrict our analysis to this normalized form. Consequently, we have $\alpha = A/(t^2 + 1)$ where $A$ is of one of the following forms:
\[
\psi_L(t-h_1)\,\psi_R(t-h_2),\qquad \psi_L\bigl( (t-h_1)(t-h_2)\bigr),\qquad \psi_R\bigl( (t-h_1)(t-h_2)\bigr).
\]
Here $\psi_L$, $\psi_R$ are defined by~\eqref{psi_LR}, and $h_1, h_2 \in \mathbb{H}$ satisfy $h_1^2 = h_2^2 = -1$. Next, we prove that any such quadratic curve factors as a product of at most two planar rotation curves, which generalizes Theorem~\ref{thm:Cayley}.

\begin{lemma}\label{lem:mid_plane}
If $A = \psi_L(t-h_1)\,\psi_R(t-h_2)$ for some $h_1, h_2 \in \mathbb{H}$ with $h_1^2=h_2^2=-1$, then there exists $U \in \operatorname{SO}_4(\mathbb{R})$ such that 
\[
\alpha = \frac{A}{t^2 + 1} = U^\top \begin{pmatrix}  
1 & 0 & 0 & 0\\
0 & 1 & 0 & 0\\
0 & 0 & \frac{t^2-1}{t^2+1} & \frac{2t}{t^2+1} \\
0 & 0 & -\frac{2t}{t^2+1} & \frac{t^2-1}{t^2+1}
\end{pmatrix} U.
\]
\end{lemma}

\begin{proof}
By Lemma~\ref{lem:elemfact}, there exist unit quaternions $a_1, a_2$ such that $a_1 h_1 \Cj{a}_1 = a_2 h_2 \Cj{a}_2 = \qi$. Then we have
\begin{align*}
\psi_L(a_1)\,A\,\psi_R(\Cj{a}_2) 
& =\psi_L\bigl(a_1(t-h_1) \bigr) \, \psi_R\bigl((t-h_2)\Cj{a}_2\bigr) \\
& =  \psi_L\bigl((t-a_1h_1\Cj{a}_1) a_1 \bigr) \, \psi_R\bigl(\Cj{a}_2(t-a_2h_2\Cj{a}_2)\bigr)\\
&=\psi_L(t- \qi) \,  \psi_L(a_1) \psi_R(\Cj{a}_2) \, \psi_R(t-\qi ).
\end{align*}
Denote $P_1 \coloneqq \psi_L(a_1)$ and $P_2 \coloneqq \psi_R(a_2)$. Then Lemma~\ref{lem:commu} implies 
\[
P_1 A P_2^\top =  \psi_L(t- \qi) P_1 P_2^\top \psi_R(t-\qi ) = P_2^\top (\psi_L(t- \qi) \psi_R(t-\qi )) P_1.
\]
A direct computation gives
\[
\psi_L(t- \qi) \psi_R(t-\qi ) = \begin{pmatrix}  
t^2+1 & 0 & 0 & 0\\
0 & t^2+1 & 0 & 0\\
0 & 0 & t^2-1 & 2t \\
0 & 0 & -2t & t^2-1
\end{pmatrix},
\]
from which we may take $U \coloneqq P_1 P_2$. 
\end{proof}
Lemma~\ref{lem:mid_plane} can be proved by a different method, and we refer the interested reader to \cite{Liye2024rational} for more details.

\begin{lemma}\label{lem:left_plane}
If $A = \psi_L\bigl((t-h_1)(t-h_2)\bigr)$ for some $h_1, h_2 \in \mathbb{H}$ with $h_1^2=h_2^2=-1$, then there exist $U,V \in \operatorname{SO}_4(\mathbb{R})$ such that 
\[
\alpha = \frac{A}{t^2 + 1} = \left[U^\top \begin{pmatrix}  
1 & 0 & 0 & 0\\
0 & 1 & 0 & 0\\
0 & 0 & \frac{t^2-1}{t^2+1} & \frac{2t}{t^2+1} \\
0 & 0 & -\frac{2t}{t^2+1} & \frac{t^2-1}{t^2+1}
\end{pmatrix}  U \right]  \left[V^\top \begin{pmatrix}  
\frac{t^2-1}{t^2+1} & \frac{2t}{t^2+1}  & 0 & 0\\
-\frac{2t}{t^2+1} & \frac{t^2-1}{t^2+1} & 0 & 0\\
0 & 0 & 1 & 0 \\
0 & 0 & 0 & 1
\end{pmatrix} V  \right].
\]
\end{lemma}

\begin{proof}
By Lemma~\ref{lem:elemfact}, there exists $a \in \mathbb{H}_1$ such that $a h_1 \Cj{a} = \qi$. Denote $h_3 \coloneqq a h_2 \Cj{a}$. We choose $b\in \mathbb{H}_1$ such that $b h_3 \Cj{b} = \qi$. Then
\[
\psi_L\bigl(a (t-h_1) (t-h_2) \Cj{a} \bigr) =  \psi_L\bigl((t- a h_1 \Cj{a})(t-a h_2\Cj{a} )\bigr) =  \psi_L\bigl((t- \qi)  (t-h_3 )\bigr).
\]
Let $U \coloneqq \psi_L(a)$ and $V \coloneqq \psi_L(b)^\top \psi_L(a)$. Note that $(t^2+1) I_4 = \psi_R(t^2+1)= \psi_R\bigl((t+ \qi)(t - \qi)\bigr)$. We deduce 
\[
(t^2 + 1)U A U^\top = \psi_R(t^2+1)\, \psi_L(t - \qi) \, \psi_L(t-h_3) =  \big[ \psi_L(t- \qi) \psi_R(t - \qi) \big] \big[ \psi_R(t + \qi) \psi_L(t-h_3) \big].
\]
Since $b h_3 \Cj{b} = \qi$, we have 
\[
\psi_R(t + \qi) \psi_L(t-h_3) = \psi_L(b) \big[ \psi_L(t- \qi) \psi_R(t + \qi) \big] \psi_L(b)^\top = UV^\top \big[ \psi_L(t- \qi) \psi_R(t + \qi) \big] VU^\top.
\]
Moreover, a direct computation implies 
\[
\psi_L(t- \qi) \psi_R(t - \qi) = \begin{pmatrix}  
t^2+1 & 0 & 0 & 0\\
0 & t^2+1 & 0 & 0\\
0 & 0 & t^2-1  & 2t \\
0 & 0 & -2t & t^2-1
\end{pmatrix},\; \psi_L(t- \qi) \psi_R(t + \qi) = \begin{pmatrix}  
t^2-1  & 2t  & 0 & 0\\
-2t & t^2-1 & 0 & 0\\
0 & 0 & t^2+1 & 0 \\
0 & 0 & 0 & t^2+1
\end{pmatrix},
\]
from which we obtain a desired decomposition of $A/(t^2 + 1)$.
\end{proof}

By the same argument, we arrive at the following characterization of quadratic curves arising from the right-right pairing.

\begin{lemma}\label{lem:right_plane}
If $A = \psi_R\bigl((t-h_1)(t-h_2)\bigr)$ for some $h_1, h_2 \in \mathbb{H}$ with $h_1^2=h_2^2=-1$, then there exist $U,V \in \operatorname{SO}_4(\mathbb{R})$ such that 
\[
\alpha = \frac{A}{t^2 + 1} = \left[U^\top \begin{pmatrix}  
1 & 0 & 0 & 0\\
0 & 1 & 0 & 0\\
0 & 0 & \frac{t^2-1}{t^2+1} & \frac{2t}{t^2+1} \\
0 & 0 & -\frac{2t}{t^2+1} & \frac{t^2-1}{t^2+1}
\end{pmatrix}  U \right]  \left[V^\top \begin{pmatrix}  
\frac{t^2-1}{t^2+1} & -\frac{2t}{t^2+1}  & 0 & 0\\
\frac{2t}{t^2+1} & \frac{t^2-1}{t^2+1} & 0 & 0\\
0 & 0 & 1 & 0 \\
0 & 0 & 0 & 1
\end{pmatrix} V  \right].
\]
\end{lemma}

Lemmas~\ref{lem:mid_plane}--\ref{lem:right_plane} show that a quadratic rational curve in $\operatorname{Rat}(4,I_4)$ is either itself a planar rotation curve, or a product of two planar rotation curves. This factorization, however, is not unique. For instance, in the proof of Lemma~\ref{lem:left_plane}, any $f \in \mathbb{H}[t]$ with $\psi_R(f) = (t^2 + 1)I_4$ may be used to yield an alternative factorization. As an immediate consequence of Theorem~\ref{thm:decomp} and Lemmas~\ref{lem:mid_plane}--\ref{lem:right_plane}, we obtain the following. 

\begin{theorem}\label{thm:futherdecomp}
Every rational curve in $\operatorname{Rat}(4,I_4)$ of degree $2d$ admits a factorization into a product of at most $2d$ planar rotation curves.
\end{theorem}

We give an explicit algorithm for Theorem~\ref{thm:futherdecomp}.

\begin{algorithm}[H]
\caption{Factorization algorithm for rational curves on $\operatorname{SO}_4(\mathbb{R})$ into planar rotation curves}
\label{alg:plane_rotation_decomp}
\begin{algorithmic}[1]

\Require A rational curve $\gamma = P/q \in \operatorname{Rat}(4,I_4)$ of degree $2d$.
\Ensure A list of at most $2d$ planar rotation curves on $\operatorname{SO}_4(\mathbb{R})$ 
with product $\gamma$.

\State $\text{Result} \gets [\,]$ \Comment{Empty list.}

\State Decompose $\gamma$ into $d$ quadratic rational curves $\gamma_k  = P_k /q_k $  (Algorithm~\ref{alg:decomp}).

\For{$k = 1$ to $d$}
    \If{$P_k  = \psi_L(t-h_1)\,\psi_R(t-h_2)$ with \(\QNorm{t-h_i}=q_k \) for \(i=1,2\)}
        \Comment{(Lemma~\ref{lem:mid_plane})}
        \State $\text{Result} \gets [\text{Result}, \gamma_k]$

    \ElsIf{$P_k  = \psi_L\bigl((t-h_1)(t-h_2)\bigr)$ with \(\QNorm{t-h_i}=q_k \) for \(i=1,2\)}
        \Comment{(Lemma~\ref{lem:left_plane})}
        \State $P_{k_1}  \gets \psi_L(t-h_1)\psi_R(t-h_1)$, $\gamma_{k_1} \gets P_{k_1} /q_k $

        \State $P_{k_2}  \gets \psi_R(t-\Cj{h}_1)\psi_L(t-h_2)$, $\gamma_{k_2} \gets P_{k_2} /q_k $

        \State $\text{Result} \gets [\text{Result}, \gamma_{k_1}, \gamma_{k_2}]$

    \ElsIf{$P_k  = \psi_R\bigl((t-h_1)(t-h_2)\bigr)$ with \(\QNorm{t-h_i}=q_k \) for \(i=1,2\)}
        \Comment{(Lemma~\ref{lem:right_plane})}
        \State $P_{k_1}  \gets \psi_R(t-h_1)\psi_L(t-h_1)$, $\gamma_{k_1} \gets P_{k_1} /q_k $

        \State $P_{k_2}  \gets \psi_L(t-\Cj{h}_1)\psi_R(t-h_2)$, $\gamma_{k_2} \gets P_{k_2} /q_k $

        \State $\text{Result} \gets [\text{Result}, \gamma_{k_1}, \gamma_{k_2}]$

    \EndIf
\EndFor

\State \Return Result 
\end{algorithmic}
\end{algorithm}

\begin{corollary}
Algorithm~\ref{alg:plane_rotation_decomp} factorizes any $\gamma \in \operatorname{Rat}(4,I_4)$ of degree $2d$ into at most $2d$ planar rotation curves.
\end{corollary}

\begin{proof}
Let $\gamma_k = P_k/q_k$. If $P_k = \psi_L(t-h_1)\,\psi_R(t-h_2)$, then $\gamma_k$ is planar rotation curve directly.  And if $P_k = \psi_L\bigl((t-h_1)(t-h_2)\bigr)$, i.e. $\gamma_k$ is left isoclinic, then 
\[
P_{k_1}P_{k_2} = \psi_L(t-h_1)\psi_R(t-h_1) \cdot \psi_R(t-\Cj{h}_1)\psi_L(t-h_2) = q_k\,\psi_L(t-h_1)\psi_L(t-h_2) = q_k \,P_k,
\]
hence $\gamma_k = \gamma_{k_1}\gamma_{k_2}$, where both factors are planar rotation curves by Lemma~\ref{lem:mid_plane}. The right isoclinic case is analogous. Thus each quadratic factor yields at most two planar rotation curves, and the total number is at most $2d$.
\end{proof}

\section{Examples}\label{sec:example}

The examples in this section illustrate Theorems~\ref{thm:decomp} and~\ref{thm:futherdecomp} through Algorithms~\ref{alg:decomp} and~\ref{alg:plane_rotation_decomp}. The first is a minimal quadratic example, included as a sanity check for the notion of a planar rotation curve. The second follows the full algorithmic pipeline: Cayley's factorization, quaternion factorization, pairing of equal-norm factors, and refinement into planar rotation curves.

\begin{example}[A quadratic planar rotation curve]
Let
\[
\gamma(t)=\frac{1}{t^2+1}
\begin{pmatrix}
t^2+1 & 0 & 0 & 0\\
0 & t^2+1 & 0 & 0\\
0 & 0 & t^2-1 & 2t\\
0 & 0 & -2t & t^2-1
\end{pmatrix}.
\]
Then $\gamma(t)$ is a rational curve on $\operatorname{SO}_4(\mathbb{R})$ of degree $2$. It fixes the $2$-plane spanned by $e_1$ and $e_2$, and rotates its orthogonal complement spanned by $e_3$ and $e_4$. Here $\{e_1,\dots, e_4\}$ is the standard basis of $\mathbb{R}^4$. In particular, Algorithm~\ref{alg:decomp} only returns one quadratic factor, and Algorithm~\ref{alg:plane_rotation_decomp} recognizes it as a single planar rotation factor.
\end{example}

\begin{example}
Let $\gamma=P/q$ be the rational curve on $\operatorname{SO}_4(\mathbb{R})$ of degree $4$, where $q \coloneqq (t^{2} + 4)(t^{2} + 2t + 5)$ and
\[
P \coloneqq \begin{pmatrix}
t^{4} + 2t^{3} + 5t^{2} + 8t + 4 & -2t^{3} - 2t^{2} - 8t - 8 & 4t^{2} + 16 & -2t^{3} - 2t^{2} - 8t - 8 \\
2t^{3} + 10t^{2} - 8 & t^{4} + 2t^{3} - 3t^{2} + 8t - 4 & -2t^{3} + 6t^{2} + 16t + 8 & 4t^{3} + 4t^{2} + 4t + 16 \\
-4t^{2} - 16 & 2t^{3} + 2t^{2} + 8t + 8 & t^{4} + 2t^{3} + 5t^{2} + 8t + 4 & -2t^{3} - 2t^{2} - 8t - 8 \\
2t^{3} - 6t^{2} - 16t - 8 & -4t^{3} - 4t^{2} - 4t - 16 & 2t^{3} + 10t^{2} - 8 & t^{4} + 2t^{3} - 3t^{2} + 8t - 4
\end{pmatrix}.
\]
We apply Algorithm~\ref{alg:decomp} to decompose $\gamma$ into a product of two quadratic rational curves.

\noindent
\textbf{Step 1 \textup{\&} 2.} Compute $\mathscr{F}(\gamma)$, $\mathscr{G}(\gamma)$, and extract common factors $\delta_f$, $\delta_g$ to obtain $L$ and $R$. This gives us 
\[
\delta_f=t^{3} + 2t^{2} + t + 8,\qquad 
\delta_g=(t^{3} + 2t^{2} + t + 8)(t^2+4),
\]
and
\[   
R =
\begin{pmatrix}
t & 0 & 2 & 0 \\
0 & t & 0 & 2 \\
-2 & 0 & t & 0 \\
0 & -2 & 0 & t
\end{pmatrix}, \quad 
L =
\begin{pmatrix}
t^{3} + 2t^{2} + t + 8 & -2t^{2} - 6t - 4 & -2t^{2} - 2 & -2t^{2} + 2t + 4 \\
2t^{2} + 6t + 4 & t^{3} + 2t^{2} + t + 8 & -2t^{2} + 2t + 4 & 2t^{2} + 2 \\
2t^{2} + 2 & 2t^{2} - 2t - 4 & t^{3} + 2t^{2} + t + 8 & -2t^{2} - 6t - 4 \\
2t^{2} - 2t - 4 & -2t^{2} - 2 & 2t^{2} + 6t + 4 & t^{3} + 2t^{2} + t + 8
\end{pmatrix}.
\]

\noindent
\textbf{Step 3.} Convert $L$ and $R$ to quaternion polynomials:
\[
    \mathcal{L} = (t^{3} + 2t^{2} + t + 8) + (2t^{2} + 6t + 4) \qi + (2t^{2} + 2)\qj + (2t^{2} - 2t - 4)\qk,\quad 
    \mathcal{R} = t + 2\qj.
\]

\noindent
\textbf{Step 4.} Factor $\mathcal{L}$ and $\mathcal{R}$ into linear factors using Algorithm~\ref{alg:quaternion}.
\[
\mathcal{L} \coloneqq \varphi_L(L) = (t + 2\qj)(t +1 + 2\qi)(t +1 + 2\qk),\quad \mathcal{R} \coloneqq \varphi_R(R) = t + 2\qj.
\]

\noindent
\textbf{Step 5.} Map back via $\psi_L$ and $\psi_R$ to obtain linear matrix polynomials $L_1,L_2,L_3$ and $R_1$ such that $P = L_1 L_2 L_3 R_1$ where
\begin{align*}
L_1 &= \begin{pmatrix}
t & 0 & -2 & 0 \\
0 & t & 0 & 2 \\
2 & 0 & t & 0 \\
0 & -2 & 0 & t
\end{pmatrix}, \quad
L_2 = \begin{pmatrix}
t + 1 & -2 & 0 & 0 \\
2 & t + 1 & 0 & 0 \\
0 & 0 & t + 1 & -2 \\
0 & 0 & 2 & t + 1
\end{pmatrix}, \\
L_3 &= \begin{pmatrix}
t + 1 & 0 & 0 & -2 \\
0 & t + 1 & -2 & 0 \\
0 & 2 & t + 1 & 0 \\
2 & 0 & 0 & t + 1
\end{pmatrix}, \quad
R_1 = \begin{pmatrix}
t & 0 & 2 & 0 \\
0 & t & 0 & 2 \\
-2 & 0 & t & 0 \\
0 & -2 & 0 & t
\end{pmatrix}.
\end{align*}

\noindent
\textbf{Step 6.} Multiply pairwise to obtain quadratic factors $\gamma_1 \coloneqq P_1/q_1$ and $\gamma_2 \coloneqq P_2/q_2$, where $q_1 =  t^2+4$, $q_2 = t^2+2t+5$ and 
\begin{align*}
  &P_1 \coloneqq  L_1R_1= \begin{pmatrix}
t^{2} + 4 & 0 & 0 & 0 \\
0 & t^{2} - 4 & 0 & 4t \\
0 & 0 & t^{2} + 4 & 0 \\
0 & -4t & 0 & t^{2} - 4
\end{pmatrix}, \\
&P_2 \coloneqq L_2L_3 = \begin{pmatrix}
t^{2} + 2t + 1 & -2t - 2 & 4 & -2t - 2 \\
2t + 2 & t^{2} + 2t + 1 & -2t - 2 & -4 \\
-4 & 2t + 2 & t^{2} + 2t + 1 & -2t - 2 \\
2t + 2 & 4 & 2t + 2 & t^{2} + 2t + 1
\end{pmatrix}.
\end{align*}
 
Then the factorization of $\gamma$ into quadratic rational curves is given by $\gamma = \gamma_1 \gamma_2$. Next, we apply Algorithm~\ref{alg:plane_rotation_decomp} to decompose $\gamma_1, \gamma_2$ into planar rotation curves.

\noindent
\textbf{Step 1.} For $\gamma_1=P_1/q_1$, since $P_1 = \psi_L(t+2\qj) \psi_R(t-2\qk)$, $\gamma_1$ is already a planar rotation curve.

\noindent
\textbf{Step 2.} For $\gamma_2=P_2/q_2$, since $P_2 = \psi_L\bigl((t+1+2\qi)(t+1+2\qk)\bigr)$, we obtain $\gamma_2 = \gamma_{21} \gamma_{21}$ where $\gamma_{21}=P_{21}/q_2$, $\gamma_{22}=P_{22}/q_2$, and
\begin{align*}
    &P_{21} = \psi_L(t+1+2\qi) \psi_R(t+1+2\qi) = \begin{pmatrix}
t^{2} + 2t - 3 & -4t - 4 & 0 & 0 \\
4t + 4 & t^{2} + 2t - 3 & 0 & 0 \\
0 & 0 & t^{2} + 2t + 5 & 0 \\
0 & 0 & 0 & t^{2} + 2t + 5
\end{pmatrix},\\[5pt]
    &P_{22} = \psi_R(t+1-2\qi) \psi_L(t+1+2\qk) = \begin{pmatrix}
t^{2} + 2t + 1 & 2t + 2 & -4 & -2t - 2 \\
-2t - 2 & t^{2} + 2t + 1 & -2t - 2 & 4 \\
-4 & 2t + 2 & t^{2} + 2t + 1 & -2t - 2 \\
2t + 2 & 4 & 2t + 2 & t^{2} + 2t + 1
\end{pmatrix}.
\end{align*}
We remark that $P_{21}$ is visibly a standard planar rotation after the shift $t\mapsto t+1$, whereas $P_{22}$ can be written as
\[
\begin{pmatrix}
0 & \frac{\sqrt{2}}{2} & 0 & \frac{\sqrt{2}}{2} \\
-\frac{\sqrt{2}}{2} & 0 & \frac{\sqrt{2}}{2} & 0 \\
0 & -\frac{\sqrt{2}}{2} & 0 & \frac{\sqrt{2}}{2} \\
-\frac{\sqrt{2}}{2} & 0 & -\frac{\sqrt{2}}{2} & 0
\end{pmatrix} 
\begin{pmatrix}
t^{2} + 2t + 5 & 0 & 0 & 0 \\
0 & t^{2} + 2t + 5 & 0 & 0 \\
0 & 0 & t^{2} + 2t - 3 & -4t - 4 \\
0 & 0 & 4t + 4 & t^{2} + 2t - 3
\end{pmatrix} 
\begin{pmatrix}
0 & -\frac{\sqrt{2}}{2} & 0 & -\frac{\sqrt{2}}{2} \\
\frac{\sqrt{2}}{2} & 0 & -\frac{\sqrt{2}}{2} & 0 \\
0 & \frac{\sqrt{2}}{2} & 0 & -\frac{\sqrt{2}}{2} \\
\frac{\sqrt{2}}{2} & 0 & \frac{\sqrt{2}}{2} & 0
\end{pmatrix}. \]
Thus, both $\gamma_{21}$ and $\gamma_{22}$ are planar rotation curves, and $\gamma = \gamma_1 \gamma_{21} \gamma_{22}$ is a factorization of $\gamma$ into three planar rotation curves.
\end{example}

\section{Conclusion and future work}\label{sec:conclusion}

We show that every rational curve on \(\operatorname{SO}_4(\mathbb{R})\) of degree \(2d\) (\(d\ge 1\)) can be decomposed into a product of \(d\) quadratic rational curves and then into a product of at most \(2d\) planar rotation curves. The proof combines two constructive ingredients: Cayley's factorization of four-dimensional rotations, which separates left and right isoclinic parts, and the factorization theory of quaternion polynomials, which reduces the isoclinic parts to linear factors. Pairing linear factors with equal norm gives the quadratic factorization, and the classification of quadratic factors completes the planar-rotation refinement. The resulting algorithms make the factorization explicit. 

From the viewpoint of applied algebraic geometry and symbolic computation, the main point is that a rational matrix curve on \(\operatorname{SO}_4(\mathbb{R})\) can be reduced to univariate quaternion polynomial factorization together with a finite pairing step. The examples illustrate how the procedure can be carried out by direct algebraic operations rather than by numerical approximation. We do not claim that the bound \(2d\) is optimal; determining the minimal number of planar rotation factors is left open. The method is specific to the special structure of \(\operatorname{SO}_4(\mathbb{R})\), especially the split \(\operatorname{Spin}(4)\cong \operatorname{Spin}(3)\times\operatorname{Spin}(3)\) and the resulting Cayley's factorization; it should not be expected to extend unchanged to other orthogonal or conformal groups.

There are several natural directions for future work. One is to incorporate translations and study rational curves on the Euclidean group \(\operatorname{SE}_4(\mathbb{R})\). Since \(\operatorname{SE}_4(\mathbb{R})\) is a semidirect product of rotations and translations, the present result should form the rotational part of such a theory, but a complete factor-count statement for \(\operatorname{SE}_4(\mathbb{R})\) requires a separate treatment of the translation component. Another direction is to extend the factorization theory to the conformal group \(\operatorname{SO}_{4,1}\). The subgroup \(\operatorname{SO}_4(\mathbb{R})\) provides a compact testing ground, while rational motions in conformal three-space naturally lead to \(\operatorname{SO}_{4,1}\) and its double cover~\cite{Kalkan2022study,Li25geometric,Dorst2019conformal}. The main challenge is that \(\operatorname{SO}_{4,1}\) is non-compact and has a richer algebraic structure, so Cayley's factorization and quaternion-factorization methods used here would need to be replaced or generalized. A useful first question is whether rational curves in that setting admit factorizations into controlled low-degree building blocks, such as planar rotation curves or hyperbolic rotations.

\bibliographystyle{plain}
\bibliography{rationalcurves}

\end{document}